\documentclass[12pt]{amsart}
\usepackage{graphicx}
\numberwithin{equation}{section} 

\newtheorem{thm}{Theorem}[section]

\newtheorem{prop}[thm]{Proposition}

\theoremstyle{definition}

\theoremstyle{remark}

\begin{document}

\title[Measurable bundles of Banach algebras]
{Measurable bundles of Banach algebras}

\author{Inomjon Ganiev}
\address{Inomjon Ganiev\\
 Department of Science in Engineering\\
Faculty of Engineering, International Islamic University Malaysia\\
P.O. Box 10, 50728\\
Kuala-Lumpur, Malaysia}  \email{{\tt inam@iium.edu.my}, {\tt
ganiev1@rambler.ru}}

\author{Farrukh Mukhamedov}
\address{Farrukh Mukhamedov\\
 Department of Computational \& Theoretical Sciences\\
Faculty of Science, International Islamic University Malaysia\\
P.O. Box, 141, 25710, Kuantan\\
Pahang, Malaysia} \email{{\tt far75m@yandex.ru}, {\tt
farrukh\_m@iium.edu.my}}

\author{Karimbergen Kudaybergenov}

\address[Karimbergen Kudaybergenov]{Department of Mathematics, Karakalpak state
university,  230113 Nukus, Uzbekistan.} \email{karim2006@mail.ru}

\begin{abstract}

In the present paper we investigate  Banach--Kantorovich
 algebras over faithful solid subalgebras of algebras measurable
 functions. We prove that any Banach--Kantorovich
 algebra over faithful solid subalgebras of algebra measurable
 functions represented as a measurable bundle of Banach algebras
 with vector-valued lifting.
We apply such representation to the spectrum of elements
Banach--Kantorovich algebras.

 \vskip 0.3cm \noindent  {\it Keywords:} measurable bundle;
Banach--Kantorovich module; lattice; Banach-Mazur Theorem.
\\

{\it AMS Subject Classification:} 47A35, 17C65, 46L70, 46L52, 28D05.
\end{abstract}

\maketitle

\section{Introduction}

It is known that the theory of
 Banach bundles  stemming  from paper  \cite{Neu},
where it is was shown that such a  theory  had vast
 applications in analysis. The study of Banach lattices in terms of sections of continuous
Banach bundles has been started by Giertz (see \cite{Ge}).  Later
Gutman \cite{Gu1} created the theory of continuous Banach bundles
and measurable Banach bundles admitting lifting \cite{Gu}. A portion
of the Gutman's theory was specified in the case of bundles of
measurable Banach lattices by Ganiev \cite{Ga} and Kusraev
\cite{K3}.

Nowadays the methods of Banach bundles has many applications in the
operator algebras \cite{Alb2, AKA, GM}. In \cite{GC1} it was
considered $ C^\ast$-algebras over ring of all measurable functions
and it has been shown that any $C^\ast$-algebra over a ring
measurable functions can be represented as a measurable bundle of
$C^\ast$-algebras. Some application of this representation to
ergodic theorems have been studied in \cite{GM2}.

It is known \cite{Sa} that one of the important results in the
theory of $C^\ast$-algebras is the Gelfand--Naimark's theorem, which
describes commutative $C^\ast$-algebras over the complex field
$\mathbb{C}$ as an algebra of complex valued continuous functions
defined on the set of all pure states of given $C^\ast$-algebra. In
\cite{CGK1} it has been proved a vector version of the
Gelfand--Naimark's theorem for commutative $C^\ast$-algebras over a
ring measurable functions. GNS-representation  for such
$C^\ast$-algebras was obtained in \cite{CGK2} .

In section 2 we consider  a Banach--Kantorovich algebra over a
faithful solid subalgebras of the algebra measurable functions. We
prove a Banach--Kantorovich algebra over a faithful solid
subalgebras represented as measurable bundle of banach algebras.
Note that in \cite{GM1} $C^*$-algebras over ideals of $L_0$ have
been considered.

In section 3 we prove  a vector version of the Gelfand--Mazur's
theorem.

\section{Measurable bundles of Banach algebras}

Let $(\Omega, \Sigma, \mu)$ be a measure space with a finite measure
$\mu$ and let $L^0(\Omega)$ be the algebra of equivalence classes of
all complex measurable functions on $\Omega.$ Let
$L^{\infty}(\Omega)$ be the algebra of all equivalence classes of
bounded complex measurable functions on $\Omega$
 with the norm
$\|\hat{f}\|_{L^{\infty}(\Omega)}= \inf\{\alpha>0: |\hat{f}|\leq
\alpha\mathbf{1}\}$, here $\mathbf{1}$ is the unit function, i.e.
$\mathbf{1}(\omega)=1$ for all $\omega\in\Omega$.

A complex linear space $X$ is said to be normed by $L^0(\Omega)$ if
there is a map $\|\cdot\|:X\rightarrow L^0(\Omega)$ such that for
any $x, y \in X,$ $\lambda\in \mathbb{C}$ the following conditions
are fulfilled:
\begin{enumerate}
\item[(1)]  $\|x\|\geq 0;$ $\|x\|=0 \Leftrightarrow x=0;$
\item[(1)]  $\|\lambda x\|=|\lambda|\|x\|;$
\item[(3)]  $\|x+y\|\leq\|x\|+\|y\|.$
\end{enumerate}
 The pair $(X, \|\cdot\|)$ is called a
\textit{lattice-normed} space over $L^0(\Omega).$  A
lattice-normed space $X$  is called $d$-\textit{decomposable} if
for any $x\in X$  with $\|x\|=\lambda_1+\lambda_2,$
$0\leq\lambda_1, \lambda_2\in L^0(\Omega),$ $\lambda_1
\lambda_2=0$ there exist $x_1, x_2 \in X$ such that $x = x_1 +
x_2$  and $\|x_k\| = \lambda_k,  k = 1, 2.$

A net $(x_\alpha)$ in $X$ is $(bo)$-converging to $x \in X,$  if
$\|x-x_\alpha\|\stackrel{(o)}{\longrightarrow}0$  (note that the
order convergence in $L^0(\Omega)$ coincides with convergence
almost everywhere). A lattice-normed space $X$ which is
$d$-decomposable and complete with respect to $(bo)$-convergence
is called a \textit{Banach–-Kantorovich space}. It is known that
every Banach–-Kantorovich space $X$  over $L^0(\Omega)$ is a
module over $L^0(\Omega)$  and $\|ax\|=|a|\|x\|$ for all $a\in
L^0(\Omega), x\in X$  (see \cite{K}).

Let $E$ be a \textit{faithful solid subalgebra} in $L^0(\Omega),$
i.e. the inequality $|x|\leq |y|$ implies $x\in E$ for arbitrary
$x\in L^0(\Omega),$  $y\in E$ and $E^\perp=\{0\}.$ Note that one has
$L^\infty\subset E\subset L^0(\Omega)$.  Consider an arbitrary
algebra $\mathcal{U}$ over the field $\mathbb{C}$ such that
$\mathcal{U}$ is a module over $E,$ i.e. $(a u)v=a(uv)=u(a v)$ for
all $a\in E,\, u,v\in\mathcal{U}.$ Consider $E$-valued norm
$\|\cdot\|$ on $\mathcal{U}$ which endows $\mathcal{U}$ with
Banach--Kanorovich structure, in particularly, one has $\|a
u\|=|a|\|u\|$ for all $a \in E,\, u\in\mathcal{U}.$

An algebra $\mathcal{U}$ is called \textit{Banach--Kantorovich
algebra} over $E,$ if for every $u, v\in\mathcal{U}$ one has $\|u
v\|\leq\|u\|\|v\|.$ If $\mathcal{U}$ is a Banach--Kantorovich
algebra over $E$ with unit $e$ such that $\|e\|=\textbf{1},$ where
$\textbf{1}$ is the unit in $E,$ then $\mathcal{U}$ is called
\textit{unital Banach--Kantorovich algebra}.

{\sc Example}. Let us provide an example of Banach--Kantorovich
algebra over $E$. To do this, let us recall some definitions taken
from \cite{CGK2}. Consider a modulus $\mathcal{A}$ over $E$, here as
before, $E$ stands for faithful solid subalgebra of $L^0(\Omega)$. A
mapping
$\langle\cdot,\cdot\rangle:\mathcal{A}\times\mathcal{A}\mapsto E$ is
called {\it $E$-valued inner product}, if for every
$x,y,z\in\mathcal{A}$, $\l\in E$ one has $\langle x,x\rangle\geq 0$;
$\langle x,x\rangle=0$ if and only if $x=0$; $\langle
x,y\rangle=\overline{\langle y,x\rangle}$;  $\langle \l
x,y\rangle=\l\langle y,x\rangle$; $\langle x+y,z\rangle=\langle
x,z\rangle+\langle y,z\rangle $.

If $\langle\cdot,\cdot\rangle:\mathcal{A}\times\mathcal{A}\mapsto E$
is a {\it $E$-valued inner product}, the formula
$\|x\|=\sqrt{\langle x,x\rangle}$ defines a $d$-decomposable
$E$-valued norm on $\mathcal{A}$. Then the pair
$(\mathcal{A},\langle\cdot,\cdot\rangle)$ is called {\it
Hilbert-Kaplansky modules}, if $(\mathcal{A},\|\cdot\|)$ is BKS over
$E$. Let $A$ and $F$ be BKS over $E$. An operator $T:A\to F$ is
called {\it $E$-linear}, if one has $T(\alpha x+\beta y)=\alpha
T(x)+\beta T(y)$ for every $x,y\in A$, $\alpha,\beta\in E$. A linear
operator $T$ is called {\it $E$-bounded} if there exists $c\in E$
such that $\|T(x)\|\leq c\|x\|$ for every $x\in A$. For $E$-linear
and $E$-bounded operator $T$ one defines $\|T\|=\sup\{\|T(x)\|:\
\|x\|\leq \mathbf{1}\}$, which is a norm of $T$ (see \cite{K}).

Now let $\mathcal{A}$ be a Hilbert-Kaplansky modulus over $E$. By
$B(\mathcal{A})$ we denote the set of $E$-linear, $E$-bounded
operators on the Hilbert-Kaplansky modules $\mathcal{A}$ over $E$.
Then $B(\mathcal{A})$ is a Banach-Kantorovich algebra over $E$.\\

We shall consider a map $\mathcal{X}:\Omega \rightarrow
X(\omega),$  where $X(\omega)\neq\{0\},$  is a Banach algebra for
all $\omega\in \Omega.$  A function $u$ is called a
\textit{section} of $\mathcal{X}$ if it is defined on $\Omega$
almost everywhere and takes a value $u(\omega)\in X(\omega)$  for
all $\omega\in \mbox{dom}\, u,$  where $\mbox{dom}\, u$ is the
domain of $u.$ Let $L$ be some set of sections.

 A pair $(\mathcal{X}, L)$ is called a
\textit{measurable  bundle of banach algebras}, if
\begin{enumerate}
\item[(1)] $\lambda_1c_1+\lambda_2c_2\in L$ for all $\lambda_1, \lambda_2\in \mathbb{C},$ $c_1,
c_2\in L,$ where $\lambda_1c_1+\lambda_2c_2:\omega\in
\mbox{dom}\,c_1\cap\mbox{dom}\,c_2\rightarrow
\lambda_1c_1(\omega)+\lambda_2c_2(\omega);$
\item[(2)] the function  $c:\omega\in
\mbox{dom}\,c\rightarrow \|c\|_{X(\omega)}$ is measurable for all
$c\in L;$
\item[(3)] the set  $\{c(\omega): c\in L, \omega\in
\mbox{dom}\,c\}$ is dense in $X(\omega)$ for all $c\in L;$
\item[(4)]  $uv\in L$ for all  $u,  v\in L,$
where $u  v:\omega\in \mbox{dom}\,(u)\cap
\mbox{dom}\,(v)\rightarrow
   u(\omega) v(\omega)$.
\end{enumerate}
A section $s$ is called \textit{simple}, if there exists $c_i, A_i
\in \Sigma,$ $i\in \overline{1, n}$ such that
$$
s(\omega)=\sum\limits_{i=1}^n \chi_{A_i}c_i.
$$
A section $u$ is called \textit{measurable} if there exists a
sequence $(s_n)_{n\in\mathbb{N}}$  of simple sections such that
$\|s_n(\omega)-u(\omega)\|_{X(\omega)}\rightarrow 0$ for almost
all $\omega\in \Omega.$
 We denote by $M(\Omega, \mathcal{X})$ the set of all
measurable sections and $L^0(\Omega, \mathcal{X})$ denotes the
factorization of this set with respect to equality almost
everywhere. By $\hat{u}$ we denote the class from $L^0(\Omega,
\mathcal{X}),$ containing section $u\in  M(\Omega, \mathcal{X}).$
A function $\omega\rightarrow \|u(u(\omega)\|_{X(\omega)}$  is
measurable for all $u\in  M(\Omega, \mathcal{X}).$  By
$\|\hat{u}\|$ we denote the element in $L^0(\Omega),$ containing
the function $\|u(\omega)\|_{X(\omega)}.$ For $\hat{u}, \hat{v}\in
L^0(\Omega)$ we put $\hat{u}\cdot\hat{v}=\widehat{u(\omega)\cdot
v(\omega)}.$

Set
$$
E(\Omega, \mathcal{X})=\{x\in L^0(\Omega, \mathcal{X}): \|x\|\in
E\}.
$$
It is known \cite{Gu}  that $E(\Omega, \mathcal{X})$ is a
Banach--Kantorovich space over $E.$ Since $X(\omega)$  is a Banach
algebra  we get
\begin{eqnarray*}
\|\hat{u} \hat{v}\|(\omega)&=&\|u(\omega)
v(\omega)\|_{X(\omega)}\leq \|u(\omega)\|_{X(\omega)}
\|v(\omega)\|_{X(\omega)}\\ &=&
\|u(\omega)\|_{X(\omega)}\|v(\omega)\|_{X(\omega)} =
\|\hat{u}\|\|\hat{v}\|(\omega)
\end{eqnarray*}
for almost all $\omega\in\Omega.$ Thus $\|\hat{u} \hat{v}\|\leq
\|\hat{u}\|\|\hat{v}\|. $ Hence, $(E(\Omega, \mathcal{X}),
\|\cdot\|)$ is a Banach--Kantorovich algebra over $E.$

So, we obtain the following

\begin{prop}\label{1}
If $\mathcal{X}$ is a measurable bundle of Banach algebras, then
$(E(\Omega, \mathcal{X}), \|\cdot\|)$  is a Banach--Kantorovich
algebra over $E.$
\end{prop}

Let $\mathcal L^{\infty}(\Omega)$ be the set of all bounded
measurable functions on $\Omega$ with the norm
$$
\|f\|_{\mathcal L^{\infty}(\Omega)}=\inf\{\alpha>0: |f(\omega)|\leq
\alpha\quad\mbox{for almost all}\quad\omega\in \Omega\}.
$$
As before, by $L^{\infty}(\Omega)$ stands for the algebra of all
equivalence classes of  bounded complex measurable functions on
$\Omega$
 with the norm
$$\|\hat{f}\|_{L^{\infty}(\Omega)}= \inf\{\alpha>0: |\hat{f}|\leq
\alpha\mathbf{1}\}.$$ Set
$$
\mathcal L^{\infty}(\Omega,\mathcal{X})=\{u\in M(\Omega,
\mathcal{X}):\|u(\omega)\|_{X(\omega)}\in \mathcal
L^{\infty}(\Omega)\}
$$
 and
 $$
 L^{\infty}(\Omega,\mathcal{X})=\{\hat{u}\in
L^0(\Omega,\mathcal{X}): \|\hat{u}\|\in L^{\infty}(\Omega)\}.
$$
One can define the spaces $\mathcal
{L^{\infty}}(\Omega,\mathcal{X})$ and
$L^{\infty}(\Omega,\mathcal{X})$ with real-valued norms
$\|u\|_{\mathcal
L^{\infty}(\Omega,\mathcal{X})}=\sup\limits_{\omega\in
\Omega}|u(\omega)|_{X(\omega)}$ and
$\|\hat{u}\|_{\infty}=\bigg\|\|\hat{u}\|\bigg\|_{L^{\infty}(\Omega)},$
respectively.

It is known \cite{Gu}, \cite{K} that there is a homomorphism
 $p:L^{\infty}(\Omega)\rightarrow\mathcal{L^{\infty}}(\Omega)$
 being
 a lifting  such that
\begin{enumerate}
\item[1.] $p(\hat{f})\in \hat{f}$ and $\mbox{dom}\,
p(\hat{f})=\Omega;$

\item[2.] $\|p(\hat{f})\|_{\mathcal
L^{\infty}(\Omega)}=\|\hat{f}\|_{L^{\infty}(\Omega)}.$
\end{enumerate}

 The homomorphism $p$ is usually called a \emph{lifting} from
$L^{\infty}(\Omega)$ to $\mathcal{L^{\infty}}(\Omega).$

 The map
$\ell_\mathcal{X}:L^{\infty}(\Omega,\mathcal{X})\rightarrow
     {\mathcal{L}}^{\infty}(\Omega,\mathcal{X})$
     is called \emph{a vector-valued lifting }(associated with $p$), if
     for all
      $\hat{u},\hat{v}\in L^{\infty}(\Omega,\mathcal{X})$ and $a\in
      L^{\infty}(\Omega)$ the following conditions are satisfied:
\begin{enumerate}
\item[1.]  $\ell_\mathcal{X}(\hat{u})\in \hat{u},\, \mbox{dom}\,(\ell_\mathcal{X}(\hat{u}))=\Omega$;

\item[2.]
$\|\ell_\mathcal{X}(\hat{u})(\omega)\|_{X(\omega)}=p(\|\hat{u}\|)(\omega);$

\item[3.] $\ell_\mathcal{X}(\hat{u}+\hat{v})=\ell_\mathcal{X}(\hat{u})+\ell_\mathcal{X}(\hat{v});$

\item[4.]  $\ell_\mathcal{X}(a \hat{u})=p(a)\ell_\mathcal{X}(\hat{u});$

\item[5.]  $\ell_\mathcal{X}(\hat{u}\hat{v})=\ell_\mathcal{X}(\hat{u})\ell_\mathcal{X}(\hat{v});$

\item[6.]  for every $\omega\in \Omega$ the set $\{\ell_\mathcal{X}(\hat{u})(\omega): \hat{u}\in L^\infty(\Omega, \mathcal{X})\}$
is dense in $X(\omega).$
\end{enumerate}

Let $\mathcal{X}$ and $\mathcal{Y}$ be measurable bundles of
banach algebras over $\Omega.$ Assume that for each
$\omega\in\Omega$ the mapping $H_{\omega}:X(\omega)\rightarrow
Y(\omega)$ is an injective homomorphism of banach algebras. A
mapping $H:\omega\rightarrow H_{\omega}$ is called
\textit{inclusion} of $\mathcal{X}$ into $\mathcal{Y}$ if one has
$$
\{H_{\omega}(u(\omega)): u\in M(\Omega,\mathcal{X})\}\subset
M(\Omega,\mathcal{Y}).
$$
If $\{H_{\omega}(u(\omega)): u\in
M(\Omega,\mathcal{X})\}=M(\Omega,\mathcal{Y})$ the inclusion $H$
is called \textit{isomorphism} from $\mathcal{X}$ onto
$\mathcal{Y}.$ In this case, the bundles $\mathcal{X}$ and
$\mathcal{Y}$ are called \textit{isomorphic}.

\begin{thm}\label{mbba}
For every Banach--Kantorovich algebra  $\mathcal{U}$ over $E$
there exists a unique (up to isomorphism) measurable bundle of
banach algebras $(\mathcal{X}, L)$ with a vector-valued lifting
$\ell_{\mathcal{X}}$ such that $\mathcal{U}$ is isometrically
isomorphic to $E(\Omega, \mathcal{X}),$ and one has
$$
\{\ell_{\mathcal{X}}(x)(\omega): x\in L^\infty(\Omega,
\mathcal{X})\}=X(\omega)
$$
for all $\omega\in\Omega.$      Moreover, if $\mathcal{U}$ is a
unital algebra, then      $X(\omega)$ is also a unital algebra for
all $\omega\in\Omega.$
\end{thm}

\begin{proof} Put
$$
\mathcal{U}_b=\{u\in\mathcal{U}: \|u\|\in L^{\infty}(\Omega)\}.
$$
It is clear that $\mathcal{U}_b$ is an $L^{\infty}(\Omega)$-module
and $(bo)$-complete in $\mathcal{U}.$

On the other hand, $\mathcal{U}_b$ is a Banach algebra
     with respect to the norm
$$
\|u\|_{\infty}=\|\|u\|\|_{L^{\infty}(\Omega)},\,u\in\mathcal{U}_b.
$$
Define a seminorm $\alpha_{\omega}$ on $\mathcal{U}_b$ by
$\alpha_{\omega}(u)=p(\|u\|)(\omega)$ for all $\omega\in\Omega,$
where $p$ is the lifting on $L^{\infty}(\Omega).$ Set
$$
I_{\omega}=\{u\in\mathcal{U}_b:\alpha_{\omega}(u)=0\}.
$$
  Let us consider  a sequence $(u_{n})_{n\in \mathbb{N}}$   in $ I_\omega$
such that  $\|u_{n}-u\|_{\infty}\rightarrow0$ for some $u\in
\mathcal{U}_b .$ Then
$$\alpha_{\omega}(u)\leq\alpha_{\omega}(u_{n}-u)+\alpha_{\omega}(u_n)\leq
\|u_{n}-u\|_{\infty}\rightarrow0.$$ Therefore, $u\in I_\omega.$ It
is clear that $\lambda_{1}u+\lambda_{2}v \in I_\omega$ for all
       $u, v\in I_\omega, \lambda_{1}, \lambda_{2}\in\mathbb{C}.$ For $u\in
       I_{\omega}, v\in \mathcal{U}_b$ one gets
$$
\alpha_{\omega}(u\cdot v)=p(\|u\cdot v\|)(\omega)\leq
       p(\|u\|)(\omega)p(\|v\|)(\omega)=\alpha_{\omega}(u)p(\|v\|)(\omega)=0,
       $$ i.e.
       $u\cdot v\in I_\omega.$ Hence,
     $I_{\omega}$ is a closed ideal in $(\mathcal{U}_b,\|\cdot\|_{\infty}).$

Let $X(\omega)=\mathcal{U}_b/I_{\omega}$ be a factor-algebra and
$i_{\omega}:\mathcal{U}_b\rightarrow X(\omega)$ be the natural
homomorphism from
     $\mathcal{U}_b$ onto $X(\omega).$ Then
$X(\omega)$ is a banach algebra with respect to the norm
      $$
      \|i_{\omega}(u)\|_{\omega}^{0}=\inf\{\|v\|_{\infty}:\ u-v\in I_{\omega}\},\,
      u\in\mathcal{U}_b, \omega\in\Omega.
      $$

Let $\|\cdot\|_{\omega}$ be the norm on $X(\omega)$ generated by
the seminorm $\alpha_{\omega},$ i.e.
$\|i_\omega(u)\|_{\omega}=\alpha_{\omega}(u),\,u\in\mathcal{U}_b$
for all $\omega\in\Omega.$

Let us show that
$$
\|i_{\omega}(u)\|_{\omega}=\|i_{\omega}(u)\|_{\omega}^{0},\,
u\in\mathcal{U}_b, \omega\in\Omega.
$$
Indeed, fix $\omega\in\Omega$ and  $u\in\mathcal{U}_b.$ If
$v\in\mathcal{U}_b$ and $u -v\in I_\omega,$ then
$$
\|i_{\omega}(u)\|_{\omega}=\alpha_{\omega}(u)\leq
\alpha_{\omega}(v)+\alpha_{\omega}(u-v)=\alpha_{\omega}(v)\leq\|v\|_{\infty}.
$$
Whence
$\|i_{\omega}(u)\|_{\omega}\leq\|i_{\omega}(u)\|_{\omega}^{0}.$

To show the converse inequality we take an arbitrary
$\varepsilon>0.$ Set
$$
A_{\varepsilon}=\{\omega'\in\Omega:p(\|u\|)(\omega')\leq\alpha_{\omega}(u)+\varepsilon\}
$$
and
$$
\pi_{\varepsilon}=\chi_{A_{\varepsilon}}, \ \
u_{\varepsilon}=\pi_{\varepsilon}u.
$$
Then
\begin{equation}\label{11}
\pi_{\varepsilon}\|u\|\leq(\alpha_{\omega}(u)+\varepsilon)\pi_{\varepsilon}
\end{equation}
and
$$
\pi_{\varepsilon}^{\perp}\|u\|\geq(\alpha_{\omega}(u)+
\varepsilon)\pi_{\varepsilon}^{\perp}.
$$
The last inequality implies that
$$
p(\pi_{\varepsilon}^{\perp}\|u\|)(\omega)\geq(\alpha_{\omega}(u)+\varepsilon)p(
\pi_{\varepsilon}^{\perp})(\omega),
$$
i.e.
$$
p(\pi_{\varepsilon}^{\perp})(\omega)(p(\|u\|)(\omega)-\alpha_{\omega}(u))\geq\varepsilon
p( \pi_{\varepsilon}^{\perp})(\omega).
$$ Hence, $0\geq \varepsilon
p(\pi_{\varepsilon}^{\perp})(\omega)$ or
$p(\pi_{\varepsilon})(\omega)=1.$ Therefore,
\begin{eqnarray*}
\alpha_{\omega}(u-u_{\varepsilon})&=&\alpha_{\omega}(\pi_{\varepsilon}^{\perp}u)=
p(\|\pi_{\varepsilon}^{\perp}u\|)=\\
&=& p(\pi_{\varepsilon}^{\perp})(\omega)p(\|u\|)(\omega)=
\alpha_{\omega}(u)=0.
\end{eqnarray*}
Consequently, $u-u_{\varepsilon}\in I_{\omega}.$ It follows from
\eqref{11} that
$$
\|u_{\varepsilon}\|=\|\pi_{\varepsilon}u\| =
\pi_{\varepsilon}\|u\|\leq(\alpha_{\omega}(u)+\varepsilon)\pi_{\varepsilon}.
$$
This means that
$\|u_{\varepsilon}\|_{\infty}\leq\alpha_{\omega}(u)+\varepsilon.$
Since  $\varepsilon>0$ be an arbitrary we get
$$
\|i_{\omega}(u)\|_{\omega}^{0}\leq\|i_{\omega}(u)\|_{\omega}+\varepsilon,
$$
i.e.
$\|i_{\omega}(u)\|_{\omega}^{0}\leq\|i_{\omega}(u)\|_{\omega}.$

Now let us define a mapping $\mathcal{X}$ which assigns for each
$\omega\in\Omega$ the  banach algebra $X(\omega).$ By $L$ we
denote the set of all sections of the form
$\omega\in\Omega:\omega\rightarrow i_{\omega}(u),$ where
$u\in\mathcal{U}_b.$ One can see that $(\mathcal{X},L)$ is a
measurable bundle of banach algebras.

Let us consider  $E(\Omega, \mathcal{X})$ with $E$-valued norm
$\|\cdot\|_{E(\Omega, \mathcal{X})}.$ Let us show that
$\mathcal{U}$ is isometrically isomorphic to $E(\Omega,
\mathcal{X}).$

For each $u\in\mathcal{U}_b$ define
$\tau(u)=\widehat{i_{\omega}(u)}.$ Then for $u\in\mathcal{U}_b$
and
 $\omega\in\Omega$ one has
$$
\|i_{\omega}(u)\|_{\omega}=\alpha_{\omega}(u)=p(\|u\|)(\omega).
$$
Hence, $\|\widehat{i_{\omega}(u)}\|_{E(\Omega,
\mathcal{X})}=\|u\|,$ and therefore, $\tau$ is an isometry from
$\mathcal{U}_b$ into $L^\infty(\Omega, \mathcal{X}).$ Since
$\tau(\mathcal{U}_b)$ contains the set of all simple sections then
$\tau$ is isometry from $\mathcal{U}_b$ onto $L^\infty(\Omega,
\mathcal{X}).$ Moreover, we have $$\tau(u\cdot
v)=\widehat{i_{\omega}(u\cdot v)}=\widehat{i_{\omega}(u)
i_{\omega}(v)}=
\widehat{i_{\omega}(u)}\widehat{i_{\omega}(v)}=\tau(u)\tau(v).$$
So, $\tau$ is an isometrically isomorphism from $\mathcal{U}_b$
onto $L^\infty(\Omega, \mathcal{X}).$ Since  $\mathcal{U}_b$ is
$(bo)$-complete in $\mathcal{U}$ we obtain that  $\tau$ can be
extended up to isometrically isomorphism from $\mathcal{U}$ onto
$E(\Omega, \mathcal{X}).$ Besides, it is clear that $\tau$
preserves the multiplication, i.e. $\tau$ is an isomorphism of
algebras $\mathcal{U}$ and $E(\Omega, \mathcal{X}).$

Now let us establish that $(\mathcal{X},L)$ is a measurable bundle
with a vector-valued lifting. Define a mapping $
\ell_{\mathcal{X}}:L^\infty(\Omega, \mathcal{X})\rightarrow
\mathcal{L}^{\infty}(\Omega,\mathcal{X}) $ by
$$
\ell_{\mathcal{X}}(\hat{u})(\omega)=i_{\omega}(\tau^{-1}(\hat{u})),\,\hat{u}\in
L^\infty(\Omega, \mathcal{X}).
$$
Since $i_{\omega}(\tau^{-1}(\hat{u}))$ is defined for all
$\omega\in\Omega,$ then $\mbox{dom}\,(l_{X}(\hat{u}))=\Omega.$ For
$\hat{u}\in L^\infty(\Omega, \mathcal{X})$ and $\omega\in\Omega$
one has
$$
p(\|\hat{u}\|)(\omega)=p(\|\tau^{-1}(\hat{u})\|)(\omega)
=\alpha_{\omega}(\tau^{-1}(\hat{u}))=
\|i_{\omega}(\tau^{-1}(\hat{u}))\|_{\omega}=\|\ell_{\mathcal{X}}(\hat{u})(\omega)\|_{\omega}.
$$

The linearity of $\ell_{\mathcal{X}}$ is evident. For $\hat{u},
\hat{v}\in L^\infty(\Omega, \mathcal{X})$ we obtain
$$
\ell_{\mathcal{X}}(\hat{u}\cdot\hat{v})(\omega)=i_{\omega}(\tau^{-1}(\hat{u}\hat{v}))=
i_{\omega}(\tau^{-1}(\hat{u}))i_{\omega}(\tau^{-1}(\hat{v}))=
\ell_{\mathcal{X}}(\hat{u})(\omega)\cdot
\ell_{\mathcal{X}}(\hat{v})(\omega).
$$
 According to the
construction one gets $
\{\ell_{\mathcal{X}}(\hat{u})(\omega):\hat{u}\in L^\infty(\Omega,
\mathcal{X})\}=X(\omega). $

Now let us prove the uniqueness of $\mathcal{X}.$ Assume that
$\mathcal{Y}$ is a measurable bundle of Banach algebras with a
vector-valued lifting $\ell_{\mathcal{Y}}$ such that
$E(\Omega,\mathcal{Y})$ is isometrically isomorphic to
$\mathcal{U}.$

Let $i$ be an isometrically isomorphism between $L^\infty(\Omega,
\mathcal{X})$ and $L^{\infty}(\Omega,\mathcal{Y}).$ Define a
linear operator
 $$
 H_{\omega}:X(\omega)\rightarrow Y(\omega)\, (\omega\in\Omega)
 $$
  by
$$
H_{\omega}(\ell_{\mathcal{X}}(\hat{u})(\omega))=\ell_{\mathcal{Y}}(i(\hat{u}))
(\omega),\,\hat{u}\in L^\infty(\Omega, \mathcal{X}).
$$
Then for $\hat{u}\in L^\infty(\Omega, \mathcal{X})$ we have
\begin{eqnarray*}
\|H_{\omega}(\ell_{\mathcal{X}}(\hat{u})(\omega))\|_{Y(\omega)}&=&
\|\ell_{\mathcal{Y}}(i(\hat{u}))(\omega)\|_{Y(\omega)}=p(\|i(\hat{u})\|)(\omega)\\
&=&p(\|\hat{u}\|)(\omega)=\|\ell_{\mathcal{X}}(\hat{u})(\omega)\|_{X(\omega)},
\end{eqnarray*}
i.e. $H_{\omega}$ is an isometry. By the same argument with
properties of vector-valued lifting one yields that $H_{\omega}$
is a homomorphism and $\{H_{\omega}(u(\omega)): u\in
M(\Omega,\mathcal{X})\}=M(\Omega,\mathcal{Y}).$ Hence,
$\mathcal{X}$ and $\mathcal{Y}$ are isometrically isomorphic.

Now assume that $e$ is a unit in $\mathcal{U},$ then
$e\in\mathcal{U}_b.$ Since $i_{\omega}:\mathcal{U}_b\rightarrow
X(\omega)$ is a homomorphism, then $e_{\omega}=i_{\omega}(e)$ is a
unit in $X(\omega)$ for all $\omega\in\Omega.$ The proof is
complete.
\end{proof}

 An operator
 $
 \Phi:\mathcal{U}\rightarrow \mathcal{U}
 $
  is called \textit{mixing preserving}
if one has
$$
\Phi\left(\sum\limits_{n=1}^{\infty}\pi_n
x_n\right)=\sum\limits_{n=1}^{\infty}\pi_n \Phi(x_n)
$$
for any sequence $(x_n)$ in $\mathcal{U}$ and partition of unity
$(\pi_n)$ in $\nabla.$

As usual, by $\mbox{Inv}(\mathcal{U})$ we denote the set of all
invertible elements of the algebra $\mathcal{U}.$ For $a, b\in E$
$a\gg b$ means that $a(\omega)>b(\omega)$ for almost all
$\omega\in \Omega.$

\begin{prop}\label{3}
Let $\mathcal{U}$ be a unital Banach--Kantorovich algebra over
$E.$ Then the following statements hold true:
\begin{enumerate}
\item[(i)] if $x\in\mathcal{U},$ $\|x\|\ll \textbf{1}$
 then the element $e-x$ is invertible
and
$$
\|(e-x)^{-1}-e\|\leq\|x\|(\textbf{1}-\|x\|)^{-1};
$$
\item[(ii)]  if $x\in \texttt{Inv}(\mathcal{U}),h\in\mathcal{U}$ and
$2\|h\|\ll\|x^{-1}\|^{-1}$  then $x+h\in \mbox{Inv}(\mathcal{U})$
and
\begin{equation}\label{12}
\|(x+h)^{-1}-x^{-1}\|\leq 2\|x^{-1}\|^{2}\|h\|;
\end{equation}
\item[(iii)] the mapping $x \in\mbox{Inv}(\mathcal{U})\rightarrow x^{-1}
    $ is continuous and mixing preserving.
\end{enumerate}
\end{prop}

\begin{proof} (i) By the inequality
$$
\sum\limits_{n=0}^{\infty}\|x^{n}\|\leq
\sum\limits_{n=0}^{\infty}\|x\|^{n}=(\textbf{1}-\|x\|)^{-1}
$$
 it follows that  the series $\sum\limits_{n=0}^{\infty}x^{n}$ $(bo)$-converges to some
    $y\in\mathcal{U}.$
The sequence
$$
(e-x)\sum\limits_{n=0}^{k}x^{n}=e-x^{k+1}
$$
simultaneously $(bo)$-converges to $(e-x)y=y(e-x)$  and $e$,
therefore, the element $y$ is inverse of $e-x.$ Furthermore, we
have
$$
\|(e-x)^{-1}-e\|=\left\|\sum\limits_{n=1}^{\infty}x^{n}\right
\|\leq\|x\|(\textbf{1}-\|x\|)^{-1}.
$$

(ii) Taking into account that  $x+h=x(e+x^{-1}h),$
$\|x^{-1}h\|\leq\|x^{-1}\|\|h\|\ll\textbf{1}$ and property (i) one
finds $x+h\in \texttt{Inv}(\mathcal{U})$ and
$\|(x+h)^{-1}-x^{-1}\|\leq 2\|x^{-1}\|^{2}\|h\|.$

 (iii) Inequality \eqref{12} implies that the mapping
$$
x\in \texttt{Inv}(\mathcal{U})\rightarrow x^{-1}\in
\texttt{Inv}(\mathcal{U})
$$
is continuous.

Let $(x_{n})_{n\in\mathbb{N}}\subset \texttt{Inv}(\mathcal{U})$
and let $(\pi_{n})_{n\in\mathbb{N}}$ be a partition of the unity
in $\nabla.$ Set
$$
x=\sum\limits_{n=1}^{\infty}\pi_{n}x_{n}.
$$
It is clear that
$$
x\left(\sum\limits_{n=1}^{\infty}\pi_{n}
x_{n}^{-1}\right)=\left(\sum\limits_{n=1}^{\infty}\pi_{n}
x_{n}^{-1}\right)x=e.
$$
Therefore  $x^{-1}=\sum\limits_{n=1}^{\infty}\pi_{n} x_{n}^{-1}.$
This yields that the mapping
 $x\in \texttt{Inv}(\mathcal{U})\rightarrow x^{-1}$  is mixing
preserving. \end{proof}

For every $x\in\mathcal{U}$ by $\mbox{sp}(x)$  we denote the set
of all $a\in E$ for which the element $a e-x$ is not invertible.

\begin{prop}\label{4}
For every $x\in\mathcal{U}\equiv E(\Omega, \mathcal{X})$ the set
$\mbox{sp}(x)$ is non-empty.
\end{prop}

\begin{proof}
Without loss of generality we can assume that $\|x\|\leq \textbf{1},$ because
$$
\mbox{sp}(x)=\{(\textbf{1}+\|x\|)a: a\in
\mbox{sp}((\textbf{1}+\|x\|)^{-1}x)\}.
$$
Then
$$
\|\ell_{\mathcal{X}}(x)(\omega)\|_{X(\omega)}\leq 1
$$
for all $\omega\in\Omega.$ Hence, for each $\omega\in\Omega$ there
exists    $\lambda_{\omega}\in\mathbb{C},
|\lambda_{\omega}|\leq 1$ such that
\begin{equation}\label{13}
\lambda_{\omega}\in \mbox{sp}(\ell_{\mathcal{X}}(x)(\omega)).
\end{equation}

Now suppose that  $\mbox{sp}(x)=\emptyset.$ Denote
   $$
   \textbf{D}=\{a\in E: |a|\leq \textbf{1}\}.
   $$
    Since the mapping
   $a\in \textbf{D}\rightarrow(a e-x)^{-1}$
   is continuous and mixing preserving, then there exists
    $$
    \sup\{\|(a e-x)^{-1}\|:a\in\textbf{D}\}=c\in E.
    $$

Now take a nonzero $\pi\in\nabla$ with $\pi c\in
L^{\infty}(\Omega).$ Then the set
$$
\Omega_0=\{\omega\in\Omega: p(\pi)(\omega)=1\}
$$ has a positive measure. Fix $\omega\in\Omega_0 .$ By
definition we have
$$
\pi\|(a e-x)^{-1}\|\leq \pi c$$ for all $a\in\textbf{D}.$
Therefore, $\pi(a e-x)^{-1}\in L^\infty(\Omega, \mathcal{X})$ for
every $a\in\textbf{D}.$ Now applying the lifting $\ell_{X}$
    to the equality $\pi(a e-x)^{-1}(a e-x)=\pi e$
     one finds
$$
\ell_{\mathcal{X}}(\pi(a e-x)^{-1})(\omega)(p(a)
    (\omega)\ell_{\mathcal{X}}(e)(\omega)-\ell_{\mathcal{X}}(x)(\omega))=
    \ell_{\mathcal{X}}(e)(\omega).
$$
This implies that the element
$p(a)(\omega)e_{\omega}-\ell_{\mathcal{X}}(x)\omega)$ is
invertible in $X(\omega)$ for all $a\in\textbf{D}.$ Due to
properties of  lifting $p$ we obtain
$$\{p(a)(\omega):a\in\textbf{D}\}=
\{\lambda_{\omega}\in\mathbb{C}:|\lambda_{\omega}|\leq1\}.
$$
So,  every $\lambda_{\omega}\in\mathbb{C}$ with
$|\lambda_{\omega}|\leq1$ does not belong to
$\mbox{sp}(\ell_{\mathcal{X}}(x)(\omega))$ for all
$\omega\in\Omega_0.$ This contradicts to \eqref{13}, which yields
the desired assertion. The proof is complete.
\end{proof}

By $\nabla$ we denote the Boolean algebra of all idempotents of
$E.$ Let $\mathcal{U}$ be a unital Banach--Kantorovich algebra
over $E,$ then the subalgebra
$$
\pi\mathcal{U}=\{\pi x: x\in\mathcal{U}\}, \pi\in\nabla,\pi\neq0
$$
is considered as unital with unit $\pi e.$

By $\mbox{spm}(x)$ we denote the set of all $a\in E$ such that for
each $\pi\in\nabla, \pi\neq0$ one has $\pi(a e-x)\notin
\mbox{Inv}(\pi\mathcal{U}).$ It is clear that $\mbox{sp}(x)\subset
\mbox{spm}(x).$

Next result is a variant of the theorem about spectrum for
elements of Banach--Kantorovich algebra over $E.$

\begin{thm}\label{5}
For every $x\in\mathcal{U}$ the set $\mbox{spm}(x)$ is nonempty,
$(o)$-closed, cyclic and bounded subset of $E.$
\end{thm}

\begin{proof} First we shall show that $\mbox{spm}(x)$
is nonempty. Indeed, for  $a\in E$ we put
$$
\nabla_{a}=\{\pi\in\nabla:\pi\neq0, \pi(a
e-x)\in\mbox{Inv}(\mathcal{\pi U})\}.
$$
Let $\pi_{a}=\bigvee\nabla_{a}.$ It is clear that $\pi_{a}(a
e-x)\in \mbox{Inv}(\pi_{a}\mathcal{U}),$ and for every
$\pi\in\nabla,\pi\leq\pi_{a}^{\perp}$ one has
\begin{equation}\label{14}
\pi(a e-x)\notin  \mbox{Inv}(\pi\mathcal{U}).
\end{equation}
Denote $\pi_{0}=\bigwedge\{\pi_{a}:a\in E\}.$ Assume that
$\pi_{0}\neq0.$ Then
$$
\pi_{0}(a e-x)\in  \mbox{Inv}(\pi_{0}\mathcal{U})
$$ for all
$a\in E.$ But this contradicts to
$\mbox{spm}_{\pi_{0}\mathcal{U}}(\pi_{0}x)\neq\emptyset.$
Therefore $\pi_{0}=0.$ Now we can choose  a sequence
$(a_{n})_{n\in\mathbb{N}}\subset E$ such that
$\bigwedge\limits_{n=1}^{\infty}\pi_{a_{n}}=0.$ Let us define
$$
q_{1}=\pi^{\perp}_{a_{1}},\,
q_{n}=\pi^{\perp}_{a_{n}}\wedge q^{\perp}_{n-1},\, n>1
$$
and put $a=\sum\limits_{n=1}^{\infty} q_{n}a_{n}.$ Then
$(q_{n})_{n\in\mathbb{N}}$ is a partition of unity in $\nabla.$
Take any $\pi\in\nabla, \pi\neq0.$ Then $\pi q_{k}\neq0$ for some
$k\in\mathbb{N}.$ From the definition of $q_{k}$ one gets $\pi
q_{k}\leq\pi^{\perp}_{a_{k}}.$ So, from \eqref{14} it follows that
$$
\pi q_{k}(a_{k} e-x)\notin  \mbox{Inv}(\pi q_{k}\mathcal{U}).
$$
The equality $\pi q_{k}a=\pi q_{k} a_{k}$ implies that
$\pi(ae-x)\notin \mbox{Inv}(\pi\mathcal{U}).$ Hence, $a\in
\mbox{spm}(x).$

Now let us show $\mbox{spm}(x)$ is cyclic. Indeed, let
$(a_{n})_{n\in\mathbb{N}}\subset \mbox{spm}(x),$ and
$(\pi_{n})_{n\in\mathbb{N}}$ be a partition of unity in $\nabla.$
Denote $a=\sum\limits_{n=1}^{\infty} \pi_{n}a_{n}.$ Take any
$\pi\in\nabla, \pi\neq0.$ Then $\pi\pi_{k}\neq0$ for some
$k\in\mathbb{N}.$ According to definition of $\pi_{k}$ we get
$$
\pi\pi_{k}(a_{k} e-x)\notin \mbox{Inv}(\pi\pi_{k}\mathcal{U}).
$$
Since $\pi\pi_{k}a=\pi\pi_{k}a_{k},$ one finds $\pi(a e-x)\notin
\mbox{Inv}(\pi\mathcal{U}).$ Hence, $a\in \mbox{spm}(x).$

To show the closedness of $\mbox{spm}(x)$ take a sequence
$(a_{n})_{n\in\mathbb{N}}\subset \mbox{spm}(x)$ such that
$a_{n}\stackrel{(o)}{\longrightarrow}a.$ Assume that $a\notin
\mbox{spm}(x).$ Then there exists $\pi\in\nabla, \pi\neq0$ such
that  $\pi(a e-x)\in \mbox{Inv}(\pi\mathcal{U}).$ From
$a_{n}\stackrel{(o)}{\longrightarrow}a,$ due to
Proposition~\ref{3}~(i) one can find $n\in\mathbb{N}$ such that
$\pi(a_{n} e-x)\in \mbox{Inv}(\pi\mathcal{U}).$ This contradicts
to $a_{n}\in \mbox{spm}(x).$ So, $a\in \mbox{spm}(x).$

Finally let us take an arbitrary element  $a\in \mbox{spm}(x).$
 Suppose  that the set
 $$
 A=\{\omega\in\Omega: |a(\omega)|>\|x\|(\omega)\}
 $$
 has a positive measure. Due
 to Proposition~\ref{3}~(i) we conclude that $\chi_{A}(a e-x)$ is invertible in
 $\chi_{A}\mathcal{U}.$ But this contradicts to $a\in \mbox{spm}(x).$ Hence
 $\chi_{A}=0$
and  $|a|\leq \|x\|,$ which implies the boundedness of
 $\mbox{spm}(x).$ The proof is complete.
    \end{proof}

\section{Applications}

Next we shall prove a vector version of Gelfand--Mazur's Theorem.

\begin{thm}\label{6}   Let $\mathcal{X}$ be a measurable bundle of banach
algebras over $\Omega$ with a lifting. If every element of the
algebra $E(\Omega,  \mathcal{X})$ with unit support is invertible,
then $E(\Omega,  \mathcal{X})$ is isometrically isomorphic to $E.$
\end{thm}

\begin{proof} Let $x\in E(\Omega,  \mathcal{X}).$ According to Theorem~\ref{5}
there exists $a_{x}\in \mbox{spm}(x).$ Let $e_{x}$ be the support
of $a_{x}e-x,$ i.e.  $e_{x}$ is an indicator function of a
measurable set $\{\omega\in\Omega:\|a_{x}e-x\|(\omega)\neq0\}.$
The element $e_{x}(a_{x}e-x)+e_{x}^{\perp}e$ has unit support, and
therefore, it is invertible, i.e. one finds $z\in E(\Omega,
\mathcal{X})$ such that
$$
(e_{x}(a_{x}e-x)+e_{x}^{\perp}e)z=e.
$$
 Whence
$e_{x}(a_{x}e-x)\in \mbox{Inv}(e_{x}E(\Omega,  \mathcal{X})).$ By
$a_{x}\in \mbox{spm}(x)$ one gets $e_{x}=0.$ This implies that
$a_{x}e-x=0,$ i.e. $a_{x}e=x.$ Due to
$$
\|x\|=|a_{x}|\|e\|=|a_{x}|,
$$ we obtain the mapping
$x\mapsto a_{x}$ is an isometry from $E(\Omega,  \mathcal{X})$
onto $E.$ For every $x,y\in E(\Omega,  \mathcal{X})$ one has
$xy=a_{x}ea_{y}e=a_{x}a_{y}e.$ Hence, the correspondence $x\mapsto
a_{x}$ is isometrically isomorphism from $E(\Omega,  \mathcal{X})$
onto $E.$ The proof is complete.
\end{proof}

Next we are going to prove an other vector version of
characterization of the field $\mathbb{C}$ in the setting Banach
algebras (see \cite[Theorem 10.19]{rud}).

\begin{thm}\label{Ma}  Let $\mathcal{X}$ be a measurable bundle of banach
algebras over $\Omega$ with a lifting. If there exists $m\in E$
such that  $\|x\|\|y\|\leq m\|xy\|$ for all $x,y\in E(\Omega,
\mathcal{X})$ then $E(\Omega, \mathcal{X})$ is isometrically
isomorphic to $E.$
\end{thm}

\begin{proof} Let us consider the following two cases.

{\sc Case 1}. Let  $m\in L^\infty(\Omega).$ Then there exists
$c\in\mathbb{R}$ such that  $m\leq c\textbf{1}$. Hence
\begin{equation*}
\|x\|\|y\|\leq c\|xy\|
\end{equation*}
for all  $x,y\in E(\Omega, X).$

Let us fix a point $\omega\in\Omega.$  Applying the lifting $p$ on
$L^{\infty}(\Omega)$ to the last inequality  we obtain
\begin{equation*}
p(\|x\|)(\omega)p(\|y\|)(\omega)\leq cp(\|xy\|)(\omega).
\end{equation*}
Taking into account this inequality  and property 6 of
$\ell_{\mathcal{X}}$ we get
\begin{equation*}
\|\ell_{\mathcal{X}}(x)(\omega)\|_{X(\omega)}\|\ell_{\mathcal{X}}(y)(\omega)\|_{X(\omega)}\leq
c\|\ell_{\mathcal{X}}(x)(\omega)\ell_{X}(y)(\omega)\|_{X(\omega)}.
\end{equation*}
This implies that
$$\|x_{\omega}\|_{X(\omega)}\|y_{\omega}\|_{X(\omega)}\leq
c\|x_{\omega}y_{\omega}\|_{X(\omega)}
$$
holds for all $x_{\omega},y_{\omega}\in X(\omega).$ According to
Theorem 10.19 \cite{rud} we conclude that $X(\omega)$ is
isomorphic to $\mathbb{C}$. Now Theorem~\ref{mbba} yields that
$E(\Omega, \mathcal{X})$ is isomorphic to $E.$ Hence, for each
$x\in E(\Omega, \mathcal{X})$ one finds $a_{x}\in E$ such that
$x=a_{x}e$. The same argument as in the proof of Theorem~\ref{Ma}
one can show that the correspondence $x\mapsto\lambda_{x}$ is
isometrically isomorphism from $E(\Omega, \mathcal{X})$ onto $E.$

{\sc Case 2.} Let  $m\in E$ be arbitrary.  Putting $x=y=e$ to the
inequality $\|x\|\|y\|\leq m\|xy\|$ implies that
$m\geq\textbf{1}.$ For each $n\in\mathbb{N}$ we put
$$
\Omega_{n}=\{\omega\in\Omega: n\leq
m(\omega)<n+1\},\,\pi_{n}=\chi_{\Omega_{n}}.
$$
Then $\bigvee\limits_{n=1}^{\infty}\pi_{n}=\textbf{1}$ and
$\pi_{n}m\leq\pi_{n}(n+1)$ for all  $n\in\mathbb{N}$. Hence, for
every  $x,y\in E(\Omega, \mathcal{X}),$  $n\in\mathbb{N}$ one has
$$
\pi_{n}\|x\|\|y\|\leq \pi_{n}(n+1)\|xy\|.
$$
The Case 1 yields that $\pi_{n} E(\Omega, \mathcal{X})$ is
isometrically isomorphic to $\pi_{n}E.$ Due to construction we
have $\bigvee\limits_{n=1}^{\infty}\pi_{n}=\textbf{1}$ and
$\pi_{i}\pi_{j}=0$ ($i\neq j$), which implies that $E(\Omega,
\mathcal{X})$ is isometrically isomorphic to $E.$ The proof is
complete.
\end{proof}

\section*{Acknowledgments}

The first (I.G) author  acknowledges  the MOHE Grant
FRGS13-071-0312. The second (F.M) and third (K.K.) named authors
 acknowledge the MOHE grant ERGS13-024-0057 for support. K.K. also thanks
International Islamic University Malaysia for kind hospitality and
providing all facilities.

\end{document}